\documentclass[reqno]{amsart}
\usepackage{amsfonts}
\usepackage{amssymb, amsfonts}

\input{xy}
\xyoption{all}

\usepackage{mathrsfs} % Ralph Smith's Formal Fonts

\newtheorem{theorem}{Theorem}[section]
\newtheorem*{theorem_main}{Main Theorem}

\newtheorem{lemma}[theorem]{Lemma}
\newtheorem{corollary}[theorem]{Corollary}

\theoremstyle{definition}

\theoremstyle{remark}
\newtheorem{remark}{Remark}

\renewcommand{\leq}{\leqslant}
\renewcommand{\geq}{\geqslant}

\newcommand{\remove}[1]{}

\newcommand{\R}{{\mathbb{R}}}

\newcommand{\Set}{{\sf Set}}
\newcommand{\uVL}{{\sf uVL}}
\newcommand{\wVL}{{\sf wVL}}
\newcommand{\uHA}{{\sf uHA}}
\newcommand{\BA}{{\sf BA}}

\DeclareMathOperator{\Spec}{\bf Spec}
\DeclareMathOperator{\Max}{\bf MaxSpec}

\DeclareMathOperator{\VV}{\mathbb{V}} % vanishing locus, zero sets
\DeclareMathOperator{\C}{\mathrm{C}} % continuous functions
\DeclareMathOperator{\LC}{\mathrm{LC}} % locally constant functions
\DeclareMathOperator{\K}{\mathrm{K}} % locally constant functions

\newcommand{\B}{\mathscr{B}}    %% Yields a BA out of a HAVL
\renewcommand{\H}{\mathscr{H}}  %% Yields a HAVL out of a BA
\newcommand{\U}{\mathscr{U}}  %% Forgetful
\newcommand{\F}{\mathscr{F}}  %% Free
\newcommand{\St}{\mathscr{S}} % Stone space

\title[Unital hyperarchimedean  vector lattices]{Unital hyperarchimedean vector lattices}

\author{Richard N.\ Ball and Vincenzo Marra}

\date{14 September 2013}

\address[Ball]{
Department of Mathematics\\
University of Denver\\
Denver, CO 80210\\
U.S.A.}

\email{rball@du.edu}

\address[Marra]{Dipartimento di Matematica ``{\it Federigo Enriques}'' \\
Universit\`{a} degli Studi di Milano \\
via Cesare Saldini 50 \\ I-20133 Milano \\
Italy}

\email{vincenzo.marra@unimi.it}


\begin{document}
%

%%%%%%%%%

  \begin{abstract} We prove that the category of unital hyperarchimedean vector lattices 
  is equivalent to the category of Boolean algebras. The key result needed to establish the equivalence is that, via the Yosida representation, such a vector lattice  is naturally isomorphic to the vector lattice of all locally constant real-valued continuous functions on a Boolean (=compact Hausdorff totally disconnected) space.  We give two applications of our main result. \end{abstract}

\maketitle

\section{Introduction}\label{s:intro}
We assume familiarity with lattice-ordered real vector spaces, known as \emph{vector lattices} or as \emph{Riesz spaces}; see  \cite{birkhoff, luxemburgzaanen}. For background  on lattice-ordered Abelian groups, see \cite{bkw, dar}.

Hyperarchimedean lattice-ordered groups and vector spaces have been intensely investigated in the last forty years. Conrad's paper \cite{conrad} includes many references to the basic results in the theory. It  has long been known (see \cite[14.1.4]{bkw}) that any lattice-ordered group of real-valued  functions, each of which has finite range,  is hyperarchimedean, but that the converse fails. Bernau showed by an explicit example \cite[Example on p.\ 41]{bernau} that the converse fails  for vector lattices, too. Conrad, however, proved that if the vector lattice has a weak order unit then  the converse does hold \cite[Corollary I]{conrad}. In this paper we prove that Conrad's representation theorem can be made functorial, by using the maximal spectral space of the vector lattice as the domain of the representation. Moreover, the functor so constructed induces a natural equivalence with Boolean algebras; this constitutes our main result, which we now state precisely.

Let $\BA$ denote the category of Boolean algebras with Boolean homomorphisms. Further, let 
$\wVL$ denote the category of unital vector lattices, i.e.,  vector lattices endowed with a distinguished \emph{weak} (\emph{order}) \emph{unit}: a member $u$ of $V^+:=\{v \in V \mid v \geq 0\}$ such that $u \wedge v = 0$ implies $v = 0$, for all $v \in V$. The $\wVL$-morphisms are the the linear lattice homomorphisms that preserve the distinguished units. Explicitly, if $(V,u)$ and $(V',u')$ are
two  objects in $\wVL$, a  morphism is a function $f \colon V \to V'$ that is a homomorphism of vector spaces, a homomorphism of lattices, and satisfies $f(u)=u'$.

A  vector lattice  $V$ is \emph{hyperarchimedean} if all of its homomorphic images are Ar\-chi\-me\-dean, the latter meaning that, for all $v,w \in V$, $0<v\leq w$ implies $nv \not \leq w$ for some positive integer $n$. It is known (see e.g.\ \cite[Thm 55.1.d]{dar}) that in any such vector lattice, a weak  unit $u$ is actually a \emph{strong} (\emph{order}) \emph{unit}, meaning that for all $v \in V^+$ there is a positive integer $n$ for which $v \leq nu$. We write $\uHA$ for the full subcategory of $\wVL$ whose objects are hyperarchimedean vector lattices. 
 
 Our main result is that $\uHA$ and $\BA$ are equivalent categories. This is the content of Section \ref{s:main}; all that goes before leads up to it, and all that comes after applies it.  As an almost immediate corollary, we show in Subsection \ref{ss:a-closed} that unital hyperarchimedean vector lattices are $a$-closed, thus reobtaining Conrad's  \cite[Theorem 5.1]{conrad}. In Subsection \ref{ss:free} we characterize the free objects in $\uHA$ with a countable generating set. This provides a vector-lattice analogue of the classical result of Tarski that the free Boolean algebra on countably many generators is the essentially unique countable atomless Boolean algebra.  
\section{Preliminaries}\label{s:pre}
 The kernels of homomorphisms of a unital vector lattice are precisely the order-convex sublattice subgroups of $K$ with the feature that  $a \wedge u \in K$ implies $a \in K$ for all $a \geq 0$. (This feature is automatic if the unit happens to be strong, as in the hyperarchimedean situation.) These kernels are often called \emph{unital $\ell$-ideals} in the literature, but we simplify this terminology because it leads to no confusion in this paper, and refer to them as \emph{ideals}. An ideal $\mathfrak{p}$ of a unital vector lattice $V$ is \emph{prime} if the quotient vector lattice $V/\mathfrak{p}$ is totally ordered. An ideal $\mathfrak{m}$ of a unital vector lattice $V$ is \emph{maximal} if the quotient vector lattice $V/\mathfrak{m}$ is totally ordered and Archimedean; equivalently, if $\mathfrak{m}$ is inclusion-maximal amongst ideals. As usual, an ideal is \emph{principal} if it is singly generated;  the principal ideal generated by $v\in V$ coincides with the set $\{w \in V \mid |w| \leq n|v| \text{ for some integer } n\geq 1\}$, where $|v|:=(v \vee 0)\vee -(v \wedge 0)$.  
\subsection{Spectra}\label{s:spectra}
For any unital vector lattice  $(V,u)$, we denote by $\Spec{V}$ the collection of prime ideals of $V$, and by $\Max{V}$ the subset of maximal ideals. 
We  topologize $\Spec{V}$ using the \emph{spectral}, or \emph{Zariski},  topology. A basis of closed sets for this topology is given by the subsets of the form
\begin{align}\label{eq:zariski}
\VV{(P)}=\{\mathfrak{p}\in \Spec{V} \,\mid\, \mathfrak{p}\supseteq P\} ,
\end{align}
as $P$ ranges over finitely generated (equivalently, principal) ideals of $V$. The resulting topological space $\Spec{V}$ is called the \emph{prime spectrum} of $V$; 
its subspace $\Max{V}$ is called the \emph{maximal spectrum} of $V$. We refer to this topology on $\Max{V}$ as the spectral topology, too. It agrees with the classical  \emph{hull-kernel} topology on rings of continuous functions \cite{gj}, \textit{mutatis mutandis}.

We shall henceforth restrict attention to vector lattices equipped with a distinguished \emph{strong}  unit. Albeit in the hyperarchimedean case weak and strong units coincide, as mentioned, in the general case a weak unit need not be strong. We write $\uVL$ for the full subcategory of $\wVL$ consisting of unital vector lattices whose distinguished weak unit is strong.

A topological space $X$ is \emph{spectral} if it is ${\rm T}_{0}$, compact, its compact open subsets form a basis for the topology that is closed under finite intersections, and $X$ is \emph{sober}, i.e.\ every non-empty irreducible closed subset of $X$ has a dense point. We recall that a closed subset of a topological space is \emph{irreducible} if it may not be  written as the union of two proper non-empty closed subsets.

\begin{lemma}\label{l:spectop}For any  vector lattice $(V,u)$ in $\uVL$, $\Spec{V}$ is a spectral space and $\Max{V}$ is a compact Hausdorff space.
\end{lemma}

\begin{proof}This is proved in \cite[10.1.4--10.1.6 and 10.2.5]{bkw} in the more general setting of  lattice-ordered Abelian groups. Although sobriety is not explicity proved, it follows by a straightforward verification.
\end{proof}
\begin{remark}We emphasize that, in the preceding lemma,  $\Spec{V}$ and $\Max{V}$ are compact precisely because of the assumption that $u$ be a strong unit of $V$. Indeed, the compact and open subsets of $\Spec{V}$ are in order-preserving bijection with the principal  ideals of $V$ \cite[10.1.3--10.1.4]{bkw}, so that $\Spec{V}$ is compact if, and only if, $V$ itself is a principal ideal if, and only if, $V$ has a strong unit.
\end{remark}

It is well known that $\Spec$ is, in fact, a functor. Let $h \colon (V,u) \to (V',u')$ be an arrow in $\uVL$. For each $\mathfrak{p}\in \Spec{V'}$, set $h^{*}(\mathfrak{p})=h^{-1}(\mathfrak{p})$. It is an exercise to check that $h^{*}(\mathfrak{p})$ is a prime ideal of $V$. Hence we obtain a function
\begin{align}\label{e:spectralmap}
h^{*}\colon \Spec{V'}\to \Spec{V}
\end{align}
called \emph{the dual  map induced by $h$}.

We record the important features of the dual map (\ref{e:spectralmap}) in Lemma \ref{l:spectralmap}, for which we recall a basic technique. We say that an ideal $J$ \emph{separates} an ideal $I$ from a filter $F$ on  $V^{+}$ if 
$I \subseteq J$ and $J \cap F = \emptyset$.

\begin{lemma}[Prime Building Lemma] In a vector lattice, every ideal can be separated from every filter disjoint from it by a prime ideal. 
\end{lemma}

\begin{proof}
For a particular ideal $I$ and filter $F \subseteq V^{+}$, $I \cap F = \emptyset$, every separating ideal is contained in a maximal separating ideal by Zorn\rq{}s Lemma.  It is then light work to show that a maximal separating ideal is prime.  
\end{proof}

\begin{lemma}\label{l:spectralmap} Let $h \colon (V,u) \to (V',u')$ be any arrow in $\uVL$, and let $h^{*}\colon \Spec{V'}\to \Spec{V}$ be the 
 dual  map induced by $h$ as in \textup{(\ref{e:spectralmap})}.
 
 \begin{enumerate}
 \item[(a)] $h^{*}$ is continuous.
 \item[(b)] $h^{*}$ restricts to a \textup{(}continuous\textup{)} map $\Max{V'}\to\Max{V}$.
 \item[(c)] $h^{*}$ is onto if, and only if, $h$ is one-one. And in this case the restriction of $h^{*}$ maps 
$\Max{V'}$ onto $\Max{V}$.
 \item[(d)] $h^{*}$ is one-one if $h$ is onto.
 \end{enumerate}
\end{lemma}

\begin{proof} (a) This follows from the fact that 
\[
h^{*-1}(\VV {(a)}) = \VV{(h(a))}, \qquad a \geq 0.
\]
Here,  $\VV {(a)}$ abbreviates the basic closed set $\VV{(\langle a\rangle)}$ as in (\ref{eq:zariski}), and $\langle a\rangle$ denotes the principal ideal of $V$ generated by $a$.

(b) Observe that an ideal is proper if and only if it omits the unit, and is maximal if and only if it is maximal with respect to doing so. It follows that $h^{-1}(\mathfrak{m})$ is a maximal ideal of $V$ if $\mathfrak{m}$ is a maximal ideal of $V'$. 

(c) If $h$ fails to be one-one, say $h(a) = h(b)$ for $a \neq b$ in $V$, then no prime of $V$ which separates these elements, i.e., contains one but not the other, can possibly lie in the range of $h^{*}$. So suppose $h$ is one-one and consider $\mathfrak{p} \in \Spec{V}$. Then $h(\mathfrak{p})$ is a sublattice subgroup of $V'$ disjoint from $ F = h(V \smallsetminus \mathfrak{p})$.   Now the convexification  of  $h(\mathfrak{p})$ in $V'$ is
\[
I = \{a \in V' \,\mid\, \exists \ b\in \mathfrak{p}^{+}\ (|a| \leq h(b))\},
\]
so that the ideal it generates is $J = \{a \,\mid\, |a| \wedge u = 0 \}$. This makes it clear that, by virtue of the fact that $h$ is one-one, $J$ is disjoint from the filter on $V'^{+}$ generated by $F\cap V^{+}$.  The Prime Building Lemma then produces a prime $\mathfrak{q} \in \Spec {V'}$ containing $J$ and disjoint from $F$. Clearly $h^{*}(\mathfrak{q}) = \mathfrak{p}$.

 (d) Let $\mathfrak{p}$ and $\mathfrak{q}$ be distinct members of $\Spec V'$, say 
$v' \in \mathfrak{p}\setminus \mathfrak{q}$. Let $v \in V$ be such that $h(v) = v'$.  Then $v \in h^{*}(\mathfrak{p}) \setminus h^{*}(\mathfrak{q})$. 
\end{proof}

\subsection{Hyperarchimedean vector lattices}\label{ss:ha}
Several characterisations of unital hyperarchimedean vector lattices  are available in the literature. We shall use two. Let us first recall that a \emph{Boolean element} of a unital vector lattice $(V,u)$ is an element  $v\in V$ such that there exist $v' \in V$ with $v\wedge v'=0$ and $v\vee v'=u$. Equivalently, $0 \leq v \leq u$ and $v \wedge (u-v) = 0$. 
\begin{lemma}\label{l:ha_char} For any  vector lattice $(V,u)$ in $\uVL$,  the following are equivalent. 
\begin{enumerate}
\item[(i)] $V$ is hyperarchimedean.
\item[(ii)] For any $v \in V$ there exists an integer $n\geq 1$ such that $w=n|v|\wedge u$ is a Boolean element.
\item[(iii)] Each prime ideal of $V$ is maximal, i.e.\ $\Spec{V}=\Max{V}$.
\end{enumerate}
\end{lemma}
\begin{proof} See \cite[Theorem 55.1]{dar}, where these equivalences are proved (in the more general setting of lattice-ordered groups) up to minor modifications. Also see \cite[6.3.2]{cdm} for proofs in the equivalent language of MV-algebras.
\end{proof}
We note the following consequence.
\begin{corollary}\label{c:ha_implies_bool} For any unital vector lattice $(V,u)$,  if $V$ is  hyperarchimedean then
$\Max{V}$ is  Boolean.
\end{corollary}
\begin{proof} Since $V$ is hypearchimedean, $u$ is a strong unit. By Lemma \ref{l:ha_char} we have $\Spec{V}=\Max{V}$, hence $\Spec{V}$ is Hausdorff by Lemma \ref{l:spectop}. But, by a standard argument, a Hausdorff spectral space is the same thing as a Boolean space  --- see,  e.g., \ \cite[Theorem 4.2]{johnstone}, where spectral spaces are called ``coherent spaces''.
\end{proof}

\subsection{The Yosida Representation}\label{ss:yosida} 
The main tool for the analysis of Archimedean unital vector lattices is the Yosida representation, which we outline here. It embeds the study of such vector lattices in the study of $\C{(X)}$, the vector lattice of continuous real-valued functions on a topological space $X$. The effect is to convert algebra into topology, and vice-versa. 

The set $\C{(X)}$ is a vector lattice under pointwise addition, scalar multiplication, and order; the function $0_{X}$ constantly equal to $0$ on $X$ is the zero element of the vector space. If $X$ is compact, the function $1_{X}$ constantly equal to $1$ over $X$ is a strong unit of $\C{(X)}$ by the Extreme Value Theorem. We  always tacitly consider $\C{(X)}$ endowed with the distinguished unit $1_{X}$. Finally, let us recall that a subset  $S\subseteq \C{(X)}$ is said to  \emph{separate  the points of $X$} if for any $x\neq y \in X$ there is $f \in S$ with $f(x)\neq f(y)$.

Whatever its antecedents, the representation was fully articulated in Yosida\rq{}s landmark paper \cite{yosida}.

\begin{theorem}[The Yosida Representation]\label{t:yosida} Let $(V, u)$ be a  vector lattice with strong unit $u$, and suppose that $V$ is Archi\-me\-dean.
\begin{enumerate}
\item[(a)]For each $\mathfrak{m}\in\Max{V}$, there exists a \emph{unique} isomorphism 
\begin{align*}%\label{e:hoelder}
\iota_{\mathfrak{m}}\colon (V/\mathfrak{m},u/\mathfrak{m})\to (\R,1)
\end{align*}
in $\uVL$. Upon setting
\begin{align*}%\label{e:hoelderfunct}
\hat{v}(\mathfrak{m}):=\iota_{\mathfrak{m}}(v/\mathfrak{m}) \in \R\,,
\end{align*}
each $v \in V$ induces a  function 
\begin{align*}
\hat{v}\colon \Max{V}\to \R
\end{align*} 
that  is continuous with respect to the spectral topology on the domain and the Euclidean topology on the co-domain. 
\item[(b)]Set $X=\Max{V}$. The map
\begin{align*}
\widehat{\cdot} \,\,\colon (V,u) \longrightarrow (\C{(X)},1_{X})
\end{align*}
given by \textup{(a)} is a monomorphism in $\uVL$ whose image $\widehat{V}\subseteq \C{(X)}$ separates the points of $X$.
\item[(c)] $X$ is unique up to homeomorphism with respect to its properties. More explicitly, if $Y$ is any compact Hausdorff space, and $e\colon (V,u)\to (\C{(Y)},1_{Y})$ is any monomorphism in $\uVL$ whose image $e(V)\subseteq \C{(Y)}$ separates the points of $Y$, then there exists a unique homeomorphism  $f : Y \to X$ such that $(e(v))(y) = \widehat{v} (f(y))$ for all $y \in Y$.
\end{enumerate} 
\end{theorem}
\section{Representation by locally constant functions}\label{s:representation}
Let $X$ and $Y$ be topological spaces. A continuous function $f\colon X \to Y$ is \emph{locally constant at $x\in X$} if there exists an open set $U\subseteq X$ containing $x$ such that $f(x')=f(x)$ for each $x' \in U$; and $f$ is \emph{locally constant} if it is locally constant at each $x\in X$. We write $\LC{(X)}$ to denote the set of all locally constant real-valued functions $X \to \R$, where the set of real numbers is endowed with its Euclidean topology; thus $\LC{(X)}\subseteq \C{(X)}$. We also write  $\K{(X)}\subseteq \C{(X)}$ for the collection of all \emph{characteristic functions} on $X$, that is,  continuous maps on $X$ taking values in the discrete subspace $\{0,1\}$ of $\R$. Under pointwise order, the set $\K{(X)}$ forms a Boolean algebra  with negation given  by $\neg \chi = 1_{X}-\chi$, for $\chi \in \K{(X)}$, where the subtraction is pointwise.

\smallskip  The following basic fact will be   often invoked. 
\begin{lemma}\label{l:rep_clopens}Let $X$ be a compact topological space, and consider $f\in \C{(X)}$. The following are equivalent.
\begin{enumerate}
\item[(i)] $f\in\LC{(X)}$.
\item[(ii)] $f$ has finite range.
\item[(iii)] There exists a unique finite partition of $X$ into clopen sets $\{C_{1},\ldots,C_{l}\}$ such that $f$ agrees with a constant function $r_{i}\in\R$ on each $C_{i}$ and, moreover, $r_{i}\neq r_{j}$ whenever $i\neq j$.
\item[(iv)] There exists a unique decomposition of $f$ into a finite linear combination $f=r_{1}\chi_{1}+\cdots+r_{k}\chi_{k}$ of characteristic functions on $X$ with $k$ minimal.
\end{enumerate}
\end{lemma}
\begin{proof} 
(i $\Leftrightarrow$ iii) To prove the non-trivial implication, assume $f \in \LC{(X)}$. For each $x \in X$, let $U_{x}$ be an open neighbourhood of $x$ over which $f$ is constant. Then $\{U_{x}\}_{x\in X}$ is an open cover of $X$. By compactness, there exists a finite subcover $\{C_{1},\ldots,C_{l}\}$; say the constant value of $f$ at $C_{i}$ is $r_{i}\in \R$, $i=1,\ldots,{l}$. Then each $C_{i}$ must be closed, hence clopen: for the complement of $C_{i}$ is the  union of open sets $\bigcup_{1\leq j\leq l \,,\,j\neq i} C_{j}$.
Next suppose $r_i=r_{j}$ and $i\neq j$. Re-indexing if necessary, we may assume $i=1$ and $j=l$. Set
$C'_{1}=C_{1}\cup C_{l}$. Since clopens are closed under finite unions,  the  collection $\{C'_{1},C_{2},\ldots,C_{l-1}\}$ is a partition of $X$ into clopen sets, and $f$ is constant over each element of the partition by construction. This proves the existence assertion. To prove uniqueness, suppose $\{C_{1},\ldots, C_{l}\}$ is as in the statement, and  $\{D_{1},\ldots, D_{l}\}$ is a further partition of $X$ into clopen sets of minimal cardinality $l$ such that $f$ takes value $s_{i}$ on each $D_{i}$, with $s_{i}\neq s_{j}$ whenever $i \neq j$. Then the range of $f$ equals both
$\{r_{1},\ldots,r_{l}\}$ and $\{s_{1},\ldots,s_{l}\}$, so that we must have $\{r_{1},\ldots,r_{l}\}=\{s_{1},\ldots,s_{l}\}$. Re-indexing if necessary, we may safely assume $r_{i}=s_{i}$ for each $i=1,\ldots, l$. Then, since $r_{i}\neq r_{j}$ if $i \neq j$, $C_{i}=f^{-1}(r_{i})=D_{i}$ holds for each $i$, as was to be shown.
 
 \smallskip \noindent The implication (iii $\Rightarrow$ ii) is immediate. To prove  (ii $\Rightarrow$ i), let us display the finite range of $f$ as $r_{1}<r_{2}<\cdots<r_{l}$. Since $\R$ is Hausdorff, we may choose disjoint open sets $I_{i}\subseteq \R$ such that $r_{i}\in I_{i}$. For each $x \in X$ there is a unique index $i$ such that $f(x)=r_{i}$. Then $f$ is locally constant on the open neighbourhood $f^{-1}(I_{i})$ of $x$, as was to be shown.
 
\smallskip \noindent (iv $\Rightarrow$ ii) Obvious.

\smallskip \noindent (iii $\Rightarrow$ iv) Let $\{C_{1},\ldots,C_{l}\}$ be a partition as in the hypothesis, and let $r_{i}$ be the value of $f$ over $C_{i}$. For each $i=1,\ldots,l$, define $\chi_{i}\colon X \to \{0,1\}$ by $\chi_{i}(x)=1$ if $x\in C_{i}$, and $\chi_{i}(x)=0$ otherwise. Then $\chi_{i}$ is continuous because $C_{i}$ is both closed and open. Obviously $f=\sum_{i=1}^{l}r_{i}\chi_{i}$. It follows that, since $r_i \neq r_j$ for $i \neq j$, the range of $f$ has cardinality $l$. If no $r_i = 0$ then $ k = l$ is palpably the minimal positive integer that affords  a decomposition of $f$ into a linear combination of characteristic functions, while if some $r_i = 0$, say $r_l = 0$, then the decomposition reduces to $f=\sum_{i=1}^{l-1}r_{i}\chi_{i}$ and $k = l-1$ is this minimal positive integer. It remains to show that if $f=\sum_{i=1}^{k}s_{i}\xi_{i}$ is another such decomposition, with $s_{i}\in \R$ and each $\xi_{i}\colon X \to \{0,1\}$ a characteristic function, then $s_{i}=r_{i}$ and $\chi_{i}=\xi_{i}$ hold to within re-indexing. Indeed, by the minimality of $k$ we must have $s_{i}\neq s_{j}$ whenever $i\neq j$. Hence the set of nonzero members of the range of $f$ equals both $\{r_1,\ldots,r_k\}$ and $\{s_1,\ldots,s_k\}$, so that $r_{i}=s_{i}$ for each $i$ upon re-indexing. If it then were the case that $\chi_{i}\neq \xi_{i}$ for some $i$, we would have $C_{i}\neq D_{i}$, from which it readily follows that there is a point $x\in X$ with $f(x)=r_{i}$ and $f(x)=r_{j}$ for $i \neq j$, a contradiction. This completes the proof.
\end{proof}
 As a first application of this basic lemma, we obtain:
\begin{lemma}\label{l:lc_implies_ha} Let $X$ be any compact topological space. Then  $\LC{(X)}$ is a  hyperarchimedean vector lattice under pointwise operations. 
Therefore, the inclusion map $\LC{(X)}\hookrightarrow \C{(X)}$ is a monomorphism of unital vector lattices if $\LC{(X)}$ is endowed with the strong unit $1_X$.  
\end{lemma}
\begin{proof} Lemma \ref{l:rep_clopens} entails that the set $\LC{(X)}$ is closed under pointwise addition, scalar multiplication, and finitary infima and suprema:
indeed, the continuous functions $f \colon X \to \R$ having finite range are clearly closed under all mentioned operations.  

To show that $\LC{(X)}$ is hyperarchimedean, let $f\colon X \to \R$ be a locally constant function. By Lemma \ref{l:ha_char} it suffices to show that $n|f| \wedge 1_{X}$ is a characteristic function for $n\geq 1$ a sufficiently large integer. We may safely assume  that $f > 0$. Using Lemma \ref{l:rep_clopens}, let us write  $f=\sum_{i=1}^k r_{i}\chi_{i}$, with $0 < r_1 < r_2 < \dots < r_k$. Choose an integer $n_0 \geq 1$ such that $n_0 r_1 \geq 1$. Then $n_{0}f \wedge 1_{X}$ takes value $1$  over $C_{i}=\chi_{i}^{-1}(1)$, $i=1,\ldots,k$, and value $0$ elsewhere, hence is a characteristic function. This completes the proof.  
\end{proof}
In light of Lemma \ref{l:lc_implies_ha}, whenever the space $X$ is compact we  always tacitly consider $\LC{(X)}$ endowed with the distinguished strong unit $1_{X}$, and hence as a subobject of $\C{(X)}$ in the category $\uVL$.

\smallskip 
The next result may perhaps be considered folklore. We provide a proof, since we did not succeed  in locating one in the literature.
\begin{lemma}\label{l:epis} Let $X$ be a Boolean space, and let $B\subseteq \K{(X)}$ be a subalgebra of the Boolean algebra of all characteristic functions on $X$. If $B$ separates the points of $X$ then $B=\K{(X)}$.
\end{lemma}
\begin{proof}By \cite[Lemma 4]{banschewskibruns} epimorphisms in the category of Boolean algebras coincide with surjective homomorphisms. It thus suffices to show that the inclusion map $B \hookrightarrow \K{(X)}$ is epic. That is to say,  given any two homomorphisms of Boolean algebras $g_1,g_2\colon \K{(X)}\to A$, if $g_{1}$ and $g_{2}$ agree on $B$, then  $g_{1}=g_{2}$. We prove the contrapositive statement. Suppose $g_{1}(\xi)\neq g_{2}(\xi)$ for some $\xi \in \K{(X)}$, but $g_{1}$ and $g_{2}$ agree on $B$. Since the 2-element Boolean algebra $\{0,1\}$ generates the variety of all Boolean algebras (see e.g.\ \cite[IV.1.12]{burrissank}), there is an onto homomorphism $h\colon A \twoheadrightarrow \{0,1\}$ that witnesses the inequality, i.e.\ $h(g_{1}(\xi))\neq h(g_{2}(\xi))$.
 Hence we may  assume $A=\{0,1\}$ without loss of generality. Then the inverse images $\mathfrak{a}=g_{1}^{-1}(0)$ and 
 $\mathfrak{b}=g_{2}^{-1}(0)$ clearly are distinct maximal ideals of $\K{(X)}$. Since $\K{(X)}$ is separating, by Theorem \ref{t:yosida} there are uniquely determined points $a\neq b \in X$ such that 
\[
\mathfrak{a}=\{\chi \in \K{(X)} \mid  \chi(a)=0\}\ \ ,\ \ \mathfrak{b}=\{\chi \in \K{(X)} \mid  \chi(b)=0\}\,,
\]
and, to within unique isomorphisms, we may  identify $g_{1}$, $g_{2}$ with 
evaluation at $a$, $b$, respectively. Since $B$ also is separating, there is $\chi \in B$ such that $\chi(a)=0$ and $\chi(b)=1$. Then
$g_{1}(\chi)=\chi(a)\neq \chi(b)=g_{2}(\chi)$, as was to be shown.
\end{proof}
\begin{lemma}\label{l:main_lemma}Let $X$ be a Boolean space, and let $V\subseteq \C{(X)}$ be a sublattice and linear subspace of $\C{(X)}$ that contains $1_{X}$, and is hyperarchimedean. The following are equivalent.
\begin{enumerate}
\item[(i)] $V$ separates the points of $X$.
\item[(ii)] $\K{(X)}\subseteq V$.
\item[(iii)] $\LC{(X)}= V$.
\end{enumerate}
\end{lemma}
\begin{proof} (ii $\Rightarrow$ i)  Let $x\neq y \in X$. Since $X$ is  Hausdorff, there are disjoint open neighbourhoods $U$ and $V$ of $x$ and $y$, respectively. Since $X$ has a basis of clopen sets, there is a clopen $C\subseteq U$ that contains $x$. The characteristic function $\chi_{C}\colon X \to \{0,1\}$ is continuous, because $C$ is clopen, and  separates $x$ from $y$.

\smallskip \noindent   (i $\Rightarrow$ ii) Suppose that $V$ separates the points of $X$.  We claim that the set $B=K{(X)}\cap V$ also does. Indeed, let $x\neq y \in X$ be separated by $f\in V$, so that $f(x)\neq f(y)$. Then
$f - f(x) \in V$ vanishes at $x$, and also separates $x$ from $y$. We may therefore safely assume that $f(x)=0$ and $f(y)\neq 0$. Now for each integer $n\geq 1$ the element $n|f|\wedge 1_{X} \in V$ also separates $x$ from $y$. Since $V$ is hyperarchimedean,
 by  Lemma \ref{l:ha_char} the function  $n|f|\wedge 1_{X}$ is a characteristic function for $n$ sufficiently large. This settles our claim. Observe that $B$ is a Boolean subalgebra of $\K{(X)}$, because $V$ is closed under meets and joins, and also under the negation operation $\neg \chi = 1_{X}-\chi$ of $\K{(X)}$. Lemma \ref{l:epis} now entails $B=\K{(X)}$.

\smallskip \noindent   (ii $\Rightarrow$ iii) If $\K{(X)}\subseteq V$ then $\LC{(X)}\subseteq V$, because $V$ is a linear space and by Lemma \ref{l:rep_clopens} each element of $\LC{(X)}$ is a linear combination of characteristic functions.

To prove the converse inclusion, we argue by contradiction. Suppose that $f\in V\setminus\LC{(X)}$; we shall show that  $V$ is not hyperarchimedean. Let $x\in X$ be such that $f$ is not locally constant at $x$. We may safely assume that $f(x)=0$, for $f-f(x)$ vanishes at $x$, lies in $V$, and is not locally constant at $x$. Now $f$  vanishes on no open neighbourhood of $x$. Consider the germinal ideal of $V$ at $x$,
\[
\mathscr{O}_{x}:=\{g \in V \mid  \text{There is an open neighbourhood $U$ of $x$ such that $g(U)=\{0\}$}  \}\,.
\]
It is elementary to verify that $\mathscr{O}_{x}$ indeed is an ideal of $V$. It is also clear that $f \not \in \mathscr{O}_{x}$ by our choice of $f$. Next consider the maximal ideal of $V$ at $x$,
\[
\mathfrak{m}_{x}:=\{g \in V \mid  g(x)=0\}\,.
\]
Again, it is clear that $\mathfrak{m}_{x}$ indeed is a maximal ideal of $V$. Obviously 
$\mathscr{O}_{x}\subseteq \mathfrak{m}_{x}$. By the Prime Building Lemma, there exists a prime ideal 
$\mathfrak{p}_{x}$ of $V$ with $\mathscr{O}_{x}\subseteq \mathfrak{p}_{x}$ that still satisfies $f\not \in \mathfrak{p}_{x}$. 
We claim that $\mathfrak{p}_{x}$ cannot be maximal. For, arguing as in the proof of (ii $\Rightarrow$ i), one sees that 
$\K{(X)}\subseteq V$ separates the points of $X$. Hence   by  Theorem \ref{t:yosida} there is a unique point $y \in X$ such that $\mathfrak{p}_{x}=\mathfrak{m}_{y}$.  Then $x \neq y$, because $f \in\mathfrak{m}_{x}$ but 
$f \not \in \mathfrak{p}_{x}$. Therefore there is $\chi \in \K{(X)}$ with $\chi(x)= 0$ and $\chi(y)= 1$, so that 
$\chi  \in \mathfrak{m}_{x}$ and $\chi  \not \in \mathfrak{m}_{y}$. But then  $\chi^{-1}(0)$ is clopen, and hence an open 
neighbourhood of $x$ over which $\chi$ vanishes, so that $\chi \in \mathscr{O}_{x}$. Since 
$\mathscr{O}_{x}\subseteq \mathfrak{p}_{x}$, we infer $\chi \in \mathfrak{p}_{x}=\mathfrak{m}_{y}$,  contradicting 
$\chi(y)=1$. Hence $\mathfrak{p}_{x}$ is a non-maximal prime, and $V$ fails to be hyperarchimedean by Lemma 
\ref{l:ha_char}. 

\smallskip \noindent   (iii $\Rightarrow$ ii) Each characteristic function on $X$ has finite range, and thus belongs to $\LC{(X)}$ by Lemma \ref{l:rep_clopens}.
\end{proof}
We can now prove the main result of this section.
\begin{theorem}\label{t:main_rep_them} Let $(V,u)$ be a unital vector lattice. The following are equivalent.
\begin{enumerate}
\item[(i)] $V$ is hyperarchimedean.
\item[(ii)] There exists a Boolean space $X$ such that $(V,u)$ is isomorphic to  $\LC{(X)}$ in $\uVL$.
\end{enumerate}
Further,  any Boolean space $X$ satisfying \textup{(ii)} is homeomorphic to $\Max{V}$, and is therefore uniquely determined by $(V,u)$ to within
a homeomorphism.
\end{theorem}
\begin{proof}(i $\Rightarrow$ ii) Set $X = \Max{V}$, and let $\widehat{\cdot} \,\,\colon V \hookrightarrow \C{(X)}$ be the unital embedding granted
by  Theorem \ref{t:yosida}. In particular, $\widehat{V}$ separates the points of $X$. Since $V$ is hyperarchimedean, $X$ is Boolean by Corollary \ref{c:ha_implies_bool}. Now $\widehat{V}=\LC{(X)}$ by Lemma \ref{l:main_lemma}.

\smallskip \noindent (ii $\Rightarrow$ i) By Lemma \ref{l:lc_implies_ha}, $\LC{(X)}$ is a hyperarchimedean vector lattice, and therefore
so is its isomorphic copy $V$.

\smallskip \noindent To prove the last two assertions, first observe that $X=\Max{V}$ in the proof of (i $\Rightarrow$ ii) above. Hence it suffices to show that if $X$ and $Y$ are Boolean spaces such that $\LC{(X)}$ and $\LC{(Y)}$ are isomorphic in $\uVL$, then $X$ is homeomorphic to $Y$. Indeed, by (iii $\Rightarrow$ i) in Lemma \ref{l:main_lemma} we know that  $\LC{(X)}$ and $\LC{(Y)}$ separate the points of $X$ and $Y$, respectively. By Theorem \ref{t:yosida} it follows that $X$ and $Y$ are homeomorphic to $\Max{\LC{(X)}}$ and $\Max{\LC{(Y)}}$,
respectively. But the latter two spaces are in turn homeomorphic because of our assumption that $\LC{(X)}$ and $\LC{(Y)}$ be isomorphic in $\uVL$. This completes the proof.
\end{proof}

\section{The functors $\B$ and $\H$}\label{s:functors}
We define functors  $\B\colon\uHA \to \BA$ and $\H\colon \BA \to \uHA$, and indicate how they  relate  to Stone duality.  We begin with $\H$.

\subsection{The Specker group of a Boolean algebra}
Let $B$ be a Boolean algebra, with greatest and least elements $\top$ and $\bot$, respectively. We say elements $b_{1},b_{2} \in B$ are \emph{disjoint} if $a\wedge b=\bot$.
By a \emph{partition} of $B$ we mean a maximal pairwise disjoint subset $P \subseteq B \smallsetminus \{\bot\}$, 
and we denote the set of finite partitions of $B$ by $F$. Partitions are ordered by refinement: 
$P_1 \leq P_2$ if for every $b_1 \in P_1$ there exists $b_2 \in P_2$ such that $b_1 \leq b_2$. 
The \emph{common refinement} of partitions  $P_1$ and  $P_2$ is 
$$
P_1 \wedge P_2 := \{b_1 \wedge b_2 \,\mid\, b_i \in P_i,\ b_1 \wedge b_2 \neq \bot \},
$$
which can readily be seen to be a partition which is finite if both $P_1$ and $P_2$ are.

The set $Q := \bigcup_{F} \R^{P}$, comprised of all maps  $v \colon P \to \R$ with $P \in F$, can be endowed with any one 
of the binary vector lattice operations $\diamond \in \{+, \wedge, \vee\}$,  as follows.  
Suppose  $v_i \in Q$, say $v_i \colon P_i \to \R$, and let $P = P_1 \wedge P_2$.  
Then $(v_1 \diamond v_2) \colon P \to \R$ is defined by declaring 
\[
(v_1 \diamond v_2)(b_1 \wedge b_2) := v_1 (b_1) \diamond v_2 (b_2).
\]
We further equip $Q$ in the obvious manner with the unary operations of scalar multiplication and with the constant $0:=z\colon \{\top\}\to \R$, where $\{\top\}$ is the unique singleton partition of $B$, and $z$ is the function given by
$z(\top)=0$.
Finally, we identify two elements $v_i \in Q$ if they agree on the common refinement of their domains, 
writing $v_1 \sim v_2$, and denoting the equivalence class of $v \in Q$ by $[v]$.  
This equivalence relation respects the operations on $Q$, and renders the quotient set
\[
\H{(B)} := Q/\sim = \{ [v] \mid v \in Q \}
\]
a vector lattice with strong unit $u = [v]$, where $v \colon \{\top\} \to \R$ with $v(\top)=1$.
 We suppress mention of the equivalence relation $\sim$ for the most part, 
abusing the terminology to the extent of saying things like ``let $v \colon P \to \R$ be a member of $\H{(B)}$''. The construct $\H{(B)}$ is known in the literature as the 
\emph{Specker group of $B$}. 
  
The operator $\H$ is functorial. If  $f \colon  B \to A$ is a Boolean homomorphism and $P$ is a finite partition of $B$ then 
\[
f(P) = \{f(b) \mid f(b) \neq \bot \}
\] 
is a finite partition of $A$.  Thus $f$ induces a $\uHA$-morphism 
$g := \H{( f)} \colon \H{(B)} \to \H{(A)}$ as follows.  Let $v \in \H{(B)}$, say  $v \colon P \to \R$ for some partition $P$ of $B$.  
Define $g(v) \colon f(P) \to \R$ by the rule 
\[
g(v)(f(b)) := v(b), \qquad b \in P.
\] 
It is straightforward to verify that $g$ is, indeed, a $\uHA$-morphism, and that $\H$ is thereby a functor $\BA \to \uHA$.

The functor $\H$ can be understood in terms of Stone duality, inasmuch as  $\H{(B)}$ is isomorphic to 
$ \LC{(\St{(B)})}$, where $\St{(B)}$ designates the Stone space of $B$. 
 Put differently, $\Max{\H{(B)}}$ is homeomorphic to $\St{(B)}$.

\subsection{The algebra of Boolean elements of a unital vector lattice}\label{ss:B}
 Let $\B{(V)}$ represent the set of Boolean elements of $(V,u)$, 
regarded as a lattice in the lattice operations inherited from $V$. (See Subsection \ref{ss:ha}). Actually, $\B{(V)}$ is a Boolean algebra 
with least element $\bot = 0$, greatest element $\top = u$, and each $v \in \B{(V)}$ complemented by $u - v$. 
We refer to $\B{(V)}$ as the \emph{algebra of Boolean elements of $V$}. 

The operator $\B$ is also functorial, since any $\uHA$ morphism $g \colon V \to W$ takes Boolean elements of $V$ to Boolean elements of $W$, so that its restriction $f := \B{(g)} \colon \B{(V)} \to \B{(W)}$ becomes a $\BA$-morphism. It is straightforward to verify that $\B$ is, indeed, a functor.  

The functor $\B$ is tied in with Stone duality, for $\B{(V)}$ is isomorphic to the Boolean algebra  of clopen subsets of $\Max{V}$.  Put differently, $\Max{V}$ is homeomorpic to $\St{(\B{(V)})}$. 

\section{The main result}\label{s:main}
We summarize our results in the Main Theorem. Its proof involves checking some routine details, and we leave those to the reader. The key point is to exhibit natural isomorphisms between the composite functors $\B{\H}$ and $\H{\B}$, and the identity functors on $\BA$ and $\uHA$, respectively.   The technical results necessary to carry this out are listed as lemmas following the proof itself.
 
\begin{theorem_main}The functors $\B\colon\uHA \to \BA$ and $\H\colon \BA \to \uHA$ are an  equivalence of categories.
\end{theorem_main}
\begin{proof}
We begin by showing that, for any Boolean algebra $B$, $\B{\H{(B)}}$ is isomorphic to $B$.  For that purpose fix $B$, and consider the map $\B{\H{(B)}} \to B$ defined by
\[
 v \longmapsto \begin{cases} v^{-1}(1) & \bot \neq v \in \B{\H{(B)}}\\
                                     \bot          &  \bot =v \in \B{\H{(B)}}
                  \end{cases}
\]            
This map is well-defined because, by Lemma \ref{l:uncomp}, the fact that $v$ is a Boolean element of $\H (B)$ 
implies that its range is contained in $\{0,1\}$, so that the domain of $v$ may be taken to be of the form 
$\{ v^{-1}(0), v^{-1}(1)\}$.  It is then straightforward to show that this map is a Boolean isomorphism. 

Now let us demonstrate that $\H{\B{(V)}}$ is isomorphic to $V$, for any unital hyperarchimedean vector lattice $(V,u)$. For that purpose fix $V$, 
and consider the map $\H{\B{(V)}} \to V$ defined by 
$$
v  \mapsto \Sigma_P v(a)a,  \qquad\H{\B{(V)}}\ni v \colon P \to \R
$$
This map is well-defined by Lemma \ref{l:lincomb_Bool}. It makes sense because the partition 
$P$ of $\B{(V)}$ is composed of Boolean elements of $V$, 
meaning that the calculation of the linear combination $\Sigma_P v(a)a$ can be performed in $V$. 
A straightforward verification shows this map to be a unital vector lattice isomorphism. 

To complete the proof it suffices to show that the two isomorphisms above are natural. This routine verification is omitted. 
\end{proof}
\begin{remark}\label{rem:bezhanishvilietal}Some of the results proved in the recent paper \cite{bezhanishvilietal} should be compared to ours.  In \cite[Section 5]{bezhanishvilietal} the authors introduce the notion of \emph{Specker $\R$-algebra}, that is to say, an $\R$-algebra which is generated (as an $\R$-algebra) by its collection of idempotent elements. This terminology is motivated by the close connection between Specker $\R$-algebras, and Specker and hyperarchimedean lattice-ordered groups, cf.\ e.g.\ \cite[Remark 6.3]{bezhanishvilietal}. In \cite[Theorem 6.8]{bezhanishvilietal} the authors prove that the category of Specker $\R$-algebras, with the $\R$-algebra homomorphisms as morphisms, is equivalent to the category of Boolean algebras. In the ring-theoretic context, this result is thus analogous to the preceding theorem.
\end{remark}
\begin{lemma}\label{l:uncomp}Let $v \colon P \to \R$ be a member of $\H{(B)}$.  Then $v$ is a Boolean element $\H{(B)}$ if, and only if, $v(P) \subseteq \{0,1\}$.
\end{lemma}
\begin{proof}
Suppose $v$ is a Boolean element of $\H{(B)}$ with complement $u - v$, and suppose 
$P = \{b_1 , b_2 , \dots ,b_n\}$ is the common refinement of the domains of $v$ and $u - v$. 
The point is that $v$ and $u - v$ may be regarded as having the same domain. 
Since $u(b_i) = 1$ for each $i$, $v(b_i)$ and $u(b_i) - v (b_i)$ are real numbers between 
$0$ and $1$ which meet to $0$ and add to $1$. It follows that both of them lie in $\{0,1\}$.

Now suppose that the range of $v \in \H{(B)}$ is contained in $\{0,1\}$, and let $P = \{b_1 , b_2 , \dots ,b_n\}$ 
be the common refinement of the domains of $v$ and $u - v$ as before.  
Since $u(b_i) = 1$ for all $i$'s, $v(b_i) = 0$ if, and only if, $u(b_i) -v(b_i) = 1$. 
It follows that $v \wedge (u - v) = 0$, i.e., that $v$ is a Boolean element of $\H{(B)}$. 
\end{proof} 
\begin{lemma}\label{l:lincomb_Bool}
Every element  $v \geq 0$ of a hyperarchimedean unital vector lattice $(V,u)$ can be uniquely expressed as a 
linear combination of Boolean elements.
\end{lemma}
\begin{proof}
By Theorem \ref{t:main_rep_them}, $V$ is isomorphic to $\LC{(X)}$ for some Boolean space $X$. 
By Lemma \ref{l:rep_clopens}, each $v \in V$ is a linear combination of characteristic functions.  
But the characteristic functions are precisely the Boolean elements of $\LC{(X)}$.  The linear combination is unique 
in the sense that any two linear combinations of minimum length expressing the same element are identical up to rearrangement.
\end{proof}
\section{Applications}\label{s:applications}
\subsection{Unital hyperarchimedean vector lattices are $a$-closed}\label{ss:a-closed}  As an instance of a notion that applies more generally,
we say that a monomorphism  $\iota\colon (V,u)\hookrightarrow (V',u')$ in $\uVL$ is  an 
\emph{$a$-extension of $(V,u)$} if the induced dual continuous map of spectral spaces 
$\iota^{*}\colon \Spec{V'}\to\Spec{V}$ is a bijection.  Since $\iota$ is one-one, 
by Lemma \ref{l:spectralmap} this is equivalent to asking that $\iota^{*}$ be one-one. 
If each $a$-extension of $(V,u)$ is an isomorphism, we say that $(V,u)$ is 
\emph{$a$-closed} (in $\uVL$). Informally, this means that no  vector lattice with strong unit $u$
that properly extends $(V,u)$ can have the same spectrum as $(V,u)$. For more on $a$-extensions, please see \cite[and references therein]{hagerkimber}.

To show that each unital hyperarchimedean vector lattice  is $a$-closed in $\uVL$, let us first record an easy observation.
\begin{lemma}\label{l:easy}Let $(V,u)$ be a unital hyperarchimedean vector lattice. If $\iota\colon (V,u)\hookrightarrow (V',u')$  is  an $a$-extension of $(V,u)$ in $\uVL$, then  $V'$  is  hyperarchimedean, too.
\end{lemma}
\begin{proof}For let $\iota^{*}\colon \Spec{V'}\to\Spec{V}$ be the induced dual continuous map of 
spectral spaces. By Lemma \ref{l:spectralmap}, $\iota^{*}$ restricts to a continuous surjection 
from $\Max{V'}$ onto $\Max{V}$. But $\Max (V) = \Spec{V}$ by Lemma \ref{l:ha_char},  
so that if $V'$ has a non-maximal prime $\mathfrak{p}\in\Spec{V'}\setminus\Max{V'}$, 
$\iota^{*}$ cannot be one-one.
\end{proof}
We can now prove:
\begin{corollary}\label{c:a-closed}Any unital hyperarchimedean vector lattice  is $a$-closed in $\uVL$.
\end{corollary}
\begin{proof}Let $\iota\colon (V,u)\hookrightarrow (V',u')$  be  an $a$-extension of $(V,u)$ in $\uVL$. By Lemma \ref{l:easy} we deduce that $V'$ is hyperarchimedean. By Theorem \ref{t:main_rep_them} we may  identify $(V',u')$ with $\LC{(X)}$, for $X=\Max{V'}=\Spec{V'}$. Now observe that $\iota(V)$ must  separate the points of $X$: for if  $\mathfrak{a}\neq \mathfrak{b} \in X$ are not separated by $\iota(V)$, then the  induced dual continuous map $\iota^{*}\colon \Max{V'}\to\Max{V}$ satisfies $\iota^{*}(\mathfrak{a})=\iota^{*}(\mathfrak{b})$ by direct inspection of its definition (\ref{e:spectralmap}). But if $\iota(V)$ is separating, Lemma \ref{l:main_lemma} implies $\iota(V)=\LC{(X)}=V'$, so that $\iota$ is onto.
\end{proof}

\subsection{Free objects, and a theorem \`a la Tarski}\label{ss:free}We write $\Set$ to denote the category of 
sets and functions. Let $\U\colon \uHA \to \Set$ be the ``forgetful'' assignment obtained by composing the functor $\B\colon \uHA\to \BA$ of Subsection \ref{ss:B} with the underlying set functor 
$\U'\colon\BA\to\Set$ that takes a Boolean algebra to its carrier set. %
\begin{corollary}\label{c:free} The functor $\U\colon \uHA \to \Set$ has a left adjoint $\F\colon \Set \to \uHA$. 
If $S$ is a set, then $\F{(S)}$ is $\LC{(\{0,1\}^{|S|})}$ up to an isomorphism in $\uHA$, 
where $\{0,1\}^{|S|}$ denotes the product of $|S|$ copies of the discrete space $\{0,1\}$, 
endowed with the product topology. Identifying each element $s\in S$ with its characteristic function 
$\chi_{s}\colon S \to \{0,1\}$, we obtain an injection $\iota \colon S \to \F{(S)}$. 
\end{corollary}
\begin{proof}As for all varieties of algebras, the underlying set functor $\U'\colon\BA\to \Set$ has a left adjoint $\F'\colon\Set \to \B$ that takes a given set  to the free Boolean algebra generated by that set.
Also, $\B$ has a left adjoint, namely $\H$, by our main theorem. Since adjoints compose, $\F:=\H\F'$ is left adjoint to $\U=\U'\B$. The rest follows from direct inspection of the definitions of the functors involved.
\end{proof}
\begin{remark}We continue to adopt the notation of the preceding corollary in this remark. As an instance of the general fact that adjointness relations can be expressed by universal arrows \cite[Theorem IV.1.2]{maclane}, we stress that for each set $S$, the object $\F{(S)}$ is uniquely determined to within a unique isomorphism in $\uHA$ 
by the following universal property of $\iota$. For each object $(V,u)$ of $\uHA$ and each function 
$f \colon S \to \B{(V)}$, there is a unique unital homomorphism of vector lattices 
$f' \colon \F{(S)}\to (V,u)$ such that $f=f'\circ\iota$. 
\end{remark}

In view of Corollary \ref{c:free}, for each cardinal $\kappa$ we write $\LC_{\kappa}$ to denote  $\LC{(\{0,1\}^{|S|})}$, where $S$ is any set of cardinality $\kappa$; and we call  $\LC_{\kappa}$ the \emph{free unital hyperarchi\-me\-dean vector lattice} on $\kappa$ generators. 

\smallskip 
If $\kappa$ is a finite integer it is easy to characterise $\LC_{\kappa}$ abstractly in $\uHA$ without using a universal property. 
\begin{corollary}\label{c:finitefree} Fix an integer $n\geq 0$. The free unital hyperarchimedean vector lattice on $n$ generators is, up to isomorphism, the unique object $(V,u)$ in $\uHA$  such that $V$ has dimension $n$ as a real vector space.
\end{corollary}
\begin{proof} It is clear that $\LC_{n}$ has linear dimension $n$. Conversely, suppose $V\cong \R^{n}$ as a linear space. By a theorem of Yudin, see e.g.\  \cite[Theorem 3.21]{aliprantis}, a cone on $\R^{n}$  determines an Archimedean lattice order on $\R^{n}$ if, and only if, it is a \emph{Yudin cone}, i.e.\ the non-negative span of a (Hamel) basis. Hence, to within an isomorphism of vector lattices, the only  order that makes $\R^{n}$ into an Archimedean vector lattice is the coordinatewise order inherited from $\R$. Direct inspection now shows that this in fact makes $\R^{n}$ into a hyperarchimedean vector lattice with unit $(1,\ldots,1)\in \R^{n}$, and hence into an isomorphic copy in $\uHA$
of $\LC_{n}$.
\end{proof}
When $\kappa=\aleph_{0}$, the first infinite cardinal,  $\{0,1\}^{\aleph_{0}}$ is known as the \emph{Cantor space}. 
By a well-known theorem of Brouwer, to within a homeomorphism the Cantor space is the unique  
non-empty compact totally disconnected metrisable space without isolated points; equivalently, 
by an application of Urysohn's metrisation theorem, the unique second-countable Boolean space 
that is non-empty and has no isolated points.  An algebraic reformulation of this fact via Stone duality 
amounts to Tarski's theorem that the free Boolean algebra on $\aleph_{0}$ generators is, to within isomorphism, 
the only non-trivial countable atomless Boolean algebra, see e.g.\ \cite[Chapter 28]{hg}. 

Our results afford a vector-lattice analogue of Tarski's theorem. We recall \cite[Chapter 3]{dar} that  
a \emph{principal polar} of a vector lattice $V$ is a set of the form
\[
P(v)=\{w \in V \,\mid\, |v|\wedge |w|=0\} 
\]
for some $v \in V$.  We further recall that the category $\uVL$  has products. While a concrete description 
of arbitrary products requires  care, finite products agree with their well-known counterparts 
in the non-unital category of vector lattices: the product of $(V,u)$ and $(V',u')$ may be concretely represented 
as $(V\times V', (u,u'))$, where $V\times V'$ is the set-theoretic Cartesian product endowed with 
pointwise operations. See e.g.\ \cite[Section 16]{dar}. Hence, when we say that $(V',u')$ is a 
\emph{direct factor} of $(V'',u')$ in $\uVL$   we mean that there exists an isomorphism 
$(V'',u'')\to (V,u)\times (V',u')$ in that category, for some $(V,u)$.
\begin{lemma}\label{l:direct_factors} Let $(V,u)$ be a unital hyperarchimedean vector lattice. 
\begin{enumerate}
 \item The Boolean elements of $(V,u)$ are in bijective correspondence with the principal polars of $V$ via the assignment $v\in\B{(V)} \mapsto P(v)$ that takes a Boolean element to the principal polar it generates.
 \item 
The Boolean elements of $(V,u)$  are in bijective correspondence with its direct factors, as follows. 
If $u'$ and $u''$ are complementary Boolean 
elements in $(V,u)$ then 
\[
(V,u)  \cong  (P(u''),u')\times (P(u'),u'').
\]
Conversely, in $(V',u')\times (V'',u'')$, $(u',0)$ and $(0,u'')$ are complementary Boolean elements, 
and $P((u',0)) = V''$ and $P((0,u'')) = V'$.  
\end{enumerate}
\end{lemma}
\begin{proof}(1)\ We need to prove that the assignment $v\in \B{(V)}\mapsto P(v)$ restricted to Boolean elements is a bijection onto principal polars. If $v\neq v'$ then either $v\wedge (u-v')\neq 0$ or  $v'\wedge (u-v)\neq 0$; say the former is the case. Then $v\wedge (u-v')\in P(v')$ but $v\wedge (u-v')\not \in P(v)$. This proves injectivity. For surjectivity, given a principal polar $P(v)$ we have $P(v)=P(n|v|\wedge u)$ for each integer $n\geq 1$. It follows by Lemma \ref{l:ha_char} that each principal polar is generated by a Boolean element.

(2) \
Suppose that $u'$ and $u''$ are complementary 
Boolean elements of $(V,u)$. Consider an arbitrary $v \in V$, and write it as a linear combination 
$\Sigma r_i u_i$ of Boolean elements $u_i$, as Lemma \ref{l:lincomb_Bool} assures you can.
Then $v' = \Sigma r_i (u_i \wedge u') \in P(u'')$ and 
$v'' = \Sigma r_i (u_i \wedge u'') \in P(u')$ satisfy $v = v' + v''$ and $|v'| \wedge |v''| = 0$.
The mapping $v \mapsto (v',v'')$ provides the desired isomorphism. 
It is straightforward to check the last sentence.  
\end{proof}
\begin{corollary}\label{c:cantorfree}The free  unital hyperarchimedean vector lattice on $\aleph_{0}$ generators is, up to isomorphism, the unique object $(V,u)$ in $\uHA$  that satisfies the following conditions.
\begin{enumerate}
\item $(V,u)$ is non-trivial, i.e.\ $u\neq 0$.
\item $(V,u)$ has no direct factor isomorphic to $(\R,1)$.
\item $V$ has countably many principal polars.
\end{enumerate}
\end{corollary}
\begin{proof}Let us first check that $\LC_{\aleph_{0}}$ satisfies (1--3). That (1) holds is clear. If (2) fails, then by a straightforward computation  the projection map onto the direct factor $(\R,1)$ has as kernel a maximal ideal which is an isolated point of $\Max{V}$. By Theorem \ref{t:yosida},  $\Max{V}$ is homeomorphic to  the Cantor space $\{0,1\}^{\aleph_{0}}$, and the latter has no isolated points. Hence (2) holds. As to (3), observe that the characteristic functions $\K{(\{0,1\}^{\aleph_{0}})}$ of the clopen subsets of $\{0,1\}^{\aleph_{0}}$ constitute the Boolean elements of $\LC_{\aleph_{0}}$; by Lemma \ref{l:direct_factors}, these correspond to the direct factors of $\LC_{\aleph_{0}}$, and to its principal polars. Since the basis of clopen sets of the Cantor space is countable, (3) holds.

Conversely, suppose $(V,u)$ has the properties listed.  By Theorem \ref{t:main_rep_them} we may assume that $V$ is 
$\LC{(X)}$ for a Boolean space $X$. Then $X$ is nonempty by (1). Let $B$ be the family  of clopen subsets of $X$. 
We claim that $B$ is the countable atomless Boolean algebra, from which it follows by Brouwer's theorem that $X$ is the Cantor space, and hence $(V,u)$ is the  free  unital hyperarchimedean vector lattice on $\aleph_{0}$ generators.

Now the characteristic functions $\K{(X)}$ of the clopen subsets of $X$ constitute the Boolean elements of $\LC{(X)}$,  
which correspond to its direct factors as in Lemma \ref{l:direct_factors}.  We conclude that $B$ is countable in light of (3). (The correspondence between elements of $B$ and direct factors of $\LC{(X)}$ can  also be spelled out  directly:
if $b$ and $c$ are complementary  clopen subsets of $X$, 
 the map  given by the rule $v \mapsto (v_{\upharpoonright b},v_{\upharpoonright c})$ is clearly an isomorphism $\LC{(X)} \to \LC{(b)} \times \LC{(c)}$. Here, $v_{\upharpoonright b}$ denotes the restriction of $v$ to $b$.)

Next suppose $B$ has an atom, $a$. Since $B$ forms a base for the topology on $X$, $a$ would have to 
be of the form $\{x\}$ for some isolated point $x\in X$.   But then $\LC{(\{x\})}=\LC_{1}\cong\R$ is a direct factor of 
$\LC{(X)}$ by the preceding paragraph, contrary to assumption (2). The proof is complete. 
\end{proof}

\providecommand{\bysame}{\leavevmode\hbox to3em{\hrulefill}\thinspace}

\end{document}